\documentclass{article}

\usepackage{color}

\definecolor{commentcolor}{rgb}{0.8,0.2627,0.0902}
\usepackage[ruled,linesnumbered]{algorithm2e}

\SetCommentSty{mycommfont}
\let\oldnl\nl
\newcommand{\nonl}{\renewcommand{\nl}{\let\nl\oldnl}}

\usepackage{epsfig} 
\usepackage{mathptmx} 
\usepackage{times} 
\usepackage{amsmath} 
\usepackage{amssymb}  
\usepackage{amsthm}

\DeclareMathOperator*{\minimize}{minimize}

\DeclareMathOperator*{\st}{subject\ to}

\usepackage[sorting = none]{biblatex}
\newtheorem{theorem}{Theorem}

\title{Structured LQ-Control of Irrigation Networks }
\author{Martin Heyden \and Richard Pates \and Anders Rantzer}

\title{A Structured Optimal Controller for Irrigation Networks}
\date{}

\author{Martin Heyden, Richard Pates and Anders Rantzer
\footnote{This work was supported by the Swedish Foundation for Strategic Research through the project SSF {RIT15-0091} SoPhy.\newline
\indent The authors are members of the LCCC Linnaeus Center and the ELLIIT Excellence Center at Lund University.\newline
\indent The authors are with the Department of Automatic Control, Lund University, Box 118, {SE-221 00 Lund}, Sweden. \newline
\indent © 2022 IEEE.  Personal use of this material is permitted.  Permission from IEEE must be obtained for all other uses, in any current or future media, including reprinting/republishing this material for advertising or promotional purposes, creating new collective works, for resale or redistribution to servers or lists, or reuse of any copyrighted component of this work in other works.
}%
}

\usepackage{xcolor}
 \usepackage{hyperref}
 \hypersetup{
   colorlinks,           
   bookmarksnumbered,
   linkcolor={blue!50!black},
   citecolor={blue!50!black}, 
   urlcolor={blue!50!black}
}
\usepackage[all]{hypcap}
\usepackage{rotating}

\usepackage[capitalise,noabbrev,nameinlink]{cleveref}
\crefname{assumption}{Assumption}{Assumptions}

\begin{document}
\maketitle

\begin{abstract}
In this paper, we apply an optimal {LQ} controller, which has an inherent structure that allows for a distributed implementation, to an irrigation network. The network consists of a water reservoir and connected water canals. The goal is to keep the levels close to the set-points when farmers take out water. The {LQ} controller is designed using a first-order approximation of the canal dynamics, while the simulation model used for evaluation uses third-order canal dynamics. The performance is compared to a P controller and an {LQ} controller designed using the third-order canal dynamics. The structured controller outperforms the {P controller}  and is close to the theoretical optimum given by the third-order {LQ} controller for disturbance rejection.

\end{abstract}

\section{Introduction}
 A large share of the available fresh water in the world is used for irrigation networks that supply water for food production. These networks are often only powered by gravity, and thus the water levels must be sufficiently high to enable transportation of the water. As a consequence, irrigation networks are often operated conservatively, as the farmers must be able to get water when they need it \cite{weyer2008control}. The efficiency of irrigation networks was estimated to be around 50\%, with half of the losses coming from large-scale distribution losses, which occur before the water reaches the farms \cite{mareels2005systems}. Improving the performance of these networks could lead to large savings in water that could allow for higher food production.

In the research literature, there are two dominant paths for controlling irrigation networks, namely local PI control \cite{weyer2002decentralised,litrico2003modelling,lozano2010simulation} and centralized {LQ} or MPC control \cite{weyer2003lq,neshastehriz2014water}.
Other approaches include distributed {LQ} \cite{lemos2012distributed} and distributed $H_\infty$-control \cite{cantoni2007control}. These distributed approaches typically use multiple iterations of communication for each sample time. An alternative is a non-iterative predictive  controller \cite{negenborn2009non}. Here the inputs are calculated sequentially by a communication sweep through the network. This implementation structure is similar to the one used in this paper.

In this work, the structured optimal {LQ} controller with a distributed implementation studied in our previous paper \cite{heyden2021structured} is applied to a model for irrigation networks. This controller combines the simple and efficient implementation of distributed methods with the performance of centralized controllers.
This {LQ} controller is synthesized using a model with first-order pool dynamics but evaluated on a model with third order-pool dynamics found in the literature. This is not a design choice as the structured {LQ} controller can only be synthesized on first-order dynamics. However, such first-order models are easier to identify.
Furthermore, the first-order pool dynamics describe the system well on slow time scales, and controllers are frequently \textit{designed} using them, see for example \cite{schuurmans1999simple,litrico2005design}. However, it is important to not excite the wave dynamics. In this paper this is achieved by applying a low-pass filter to the measurements taken at each gate. This means that the controller can be designed based on a first-order model in conjunction with knowledge of the dominant wave frequency.

Our contributions are twofold. Firstly, we show how to apply the structured controller in \cite{heyden2021structured} to irrigation models based on third-order canal dynamics. Secondly, we compare the performance to a simple P controller and an {LQ} controller with full state knowledge synthesized using the third-order canal dynamics. The P controller gives a baseline for easily achievable performance while the {LQ} controller gives optimal performance. 
For disturbance rejection of low-pass filtered disturbances, the structured controller is very close to the best performance and outperforms the P controller. For a change in set-points, the structured controller is in-between the maximum performance and the performance of the P controller.

\section{Problem Description}\label{sec:problem_desc}
\label{sec:problem_description}
Irrigation networks consist of a set of canals (often called pools), gates, and off-takes. The canals are connected with gates that allow for the flow between the canals to be regulated. The off-takes, often located at the gates, allow water to be taken from the canal to a farmer. The gates and off-takes are typically only powered by gravity, and thus the levels at the gates and off-takes must be sufficiently high to allow the water to be transported.
Many irrigation networks are located in rural areas, where both communication and computational capabilities are limited.

When controlling irrigation networks, there are typically several objectives that are considered \cite{weyer2008control}. The first is to keep the canal levels close to the set-points to allow the off-takes to be used.
The second is to minimize gate movement in order to reduce wear and tear, and minimize energy consumption.
Finally it is also common to try to minimize the flow over the last gate to reduce water wastage.
At the same time, the controller must handle the disturbances due to the off-takes.

To model a string of $N$ pools, we assume we measure the levels $y_1,\ y_2,\ ...\ , y_N$ relative to a nominal value at the end of each pool.  Each pool $i$ is affected by an inflow $u_i$, an outflow $u_{i-1}$, and a disturbance $d_i$. The flows $u_i$ between two pools are also relative to a nominal flow. For a schematic of the system, see Fig.~\ref{fig:schematic}.
The disturbance $d_i$ is the off-take to the farm(s) at gate $i$. We assume that these disturbances are planned, that is the controller knows, but cannot change, the value of $d_i[t]$. This means that the farmers must tell the irrigation network controller in advance that they will take out water.
To keep the indexing consistent with our previous work, we denote the most upstream canal as canal $N$. This canal has inflow from a reservoir with a capacity so large it can be assumed to be infinite for the purposes of regulation. 
\begin{figure}
    \centering
    \includegraphics{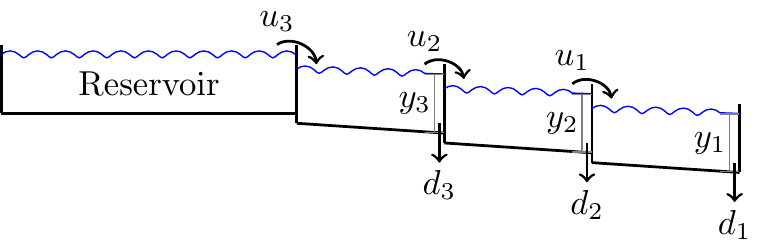}
    \caption{Graphical illustration of the problem considered. At the top of the network is a reservoir with infinite capacity. Each pool $i$ has an inflow $u_i$, an outflow $u_{i-1}$ (except pool 1) and a disturbance $d_i$ which takes water out of the pool. The goal is to regulate the water level $y_i$ at the gates.}
    \label{fig:schematic}
\end{figure}
Finally, we let the flow over the last gate be fixed. This is possible if the level in the last pool is kept close to the set-point, and doing so is highly desirable as the flow over the last gate leaves the system and can not be utilized \cite{cantoni2007control} (typically this would be fixed to be as low as possible).

Next, we will describe the simulation models used, including the dynamics for each pool in the system. The section is then concluded with a presentation of the performance criterion used.

\subsection{Network Model}
In this paper two different types of pool model, for two different pools, are used (a total of four models). For evaluation, we use third-order models found using system identification on pools 9 and 10 in the Haughton main river. See \cite{ooi2005physical} for the origin of the parameters, where it was also shown that the  models are as accurate as a PDE approach using the St-Venant equations. First order approximations of these models are used for controller design (as will be discussed in detail later). We use the two pool models to construct networks containing multiple pools. The first network type is a non-homogeneous network which alternates between the first and the second pool model. This network model is used to assess the effect of heterogeneity. The second network type is a homogeneous network using only the first pool model. This network is suitable to clearly see the effect of, for example, changing the size of the network.
 
Two modifications to the original pool models are made. Firstly, in \cite{ooi2005physical} the flow over a gate $i$ is in the form of $(y_i-p_i)^{3/2}$ where $p$ is the position of the gate relative to the nominal water level. This non-linearity can be canceled out (see for example \cite{cantoni2007control}) by letting $u_i = (y_i-p_i)^{3/2}$.  Secondly, we expand the pool models with a disturbance corresponding to an off-take. The assumption is that the off-take takes water out of the pool in the same way as the outflow.  The modified pool dynamics are in the form of
\begin{equation}\label{eq:third_order}
\begin{aligned}
y_i[t+1]  =  b_{i,1}u_i[t-\tau_i]- b_{i,2}u_i[t-\tau_i-1] + b_{i,3}u_i[t-\tau_i-2]& \\
  - c_{i,1}(u_{i-1}[t]-d_i[t]) + c_{i,2}(u_{i-1}[t-1]-d_i[t-1])& \\-c_{i,3}(u_{i-1}[t-2]-d_i[t-2])& \\
  + y_i[t] + \alpha_{i,1}(y_i[t] - 2y_i[t-1]+y_i[t-2])
  + \alpha_{i,2}(y_i[t] - y_i[t-1])&.
  \end{aligned}
\end{equation}
The sample time is one minute and the parameters for the two pools can be found in Table~\ref{tab:parameters}.  

\begin{table*}
\centering
\caption{The parameters for the first and third-order models. For the first-order model we let $b_i = b_{i,1}$ and $c_{i} = c_{i,1}$. The sample time is one minute.}
\label{tab:parameters}
\begin{tabular}{l l l l l l l l l l l}
\hline
Pool & Order & $b_{i,1}$ & $b_{i,2}$ & $b_{i,3}$ & $c_{i,1}$ & $c_{i,2}$ & $c_{i,3}$ & $\alpha_{i,1}$ & $\alpha_{i,2}$ & $\tau_i$ \\
\hline
1 & 1 & 0.069 & & & 0.063 & & & & & 3\\
1 & 3 & 0.137 & 0.155 & 0.053 & 0.190 & 0.333 & 0.175& 0.978 & 0.468 & 3\\
2 & 1 & 0.0213 & & & 0.0156 & & & & & 14\\
2 & 3 & 0.134 & 0.244 & 0.114 & 0.101 & 0.185 & 0.087& 0.314 & 0.814 & 16\\
\hline
\end{tabular}

\end{table*}

\subsection{Performance Evaluation}
The performance of the system is measured as the deviation from the nominal values for the levels $y_i$ and flows $u_i$, and how much the input changes, that is $(u_i[t+1] - u_i[t])^2$. This is done by considering the cost
\begin{equation}\label{eq:cost}
\sum_{t= 0}^\infty \sum_{i=1}^N \Big(q_i y_i[t]^2 + r_i u_i[t]^2 + \rho_i(u_i[t+1]-u_i[t])^2 \Big).
\end{equation}
The reason for penalizing $(u_i[t+1]-u_i[t])^2$ is twofold. Firstly it penalizes the wear and tear of the actuator. Secondly, it reduces the energy consumption. The amount of energy available can be limited, for example when the only available energy comes from solar power.
The structured controller can only be used when $\rho_i = 0$ for all inputs and $r_i = 0$ for all inputs except for $i=N$, which is the flow out from the reservoir into pool $N$. The effect of these limitations will be explored in the simulation section.

\section{A Structured Optimal Controller for a First-Order System}\label{sec:struct_contr}
As previously discussed, a controller for an irrigation network must handle the disturbances from the off-takes. If the network is in a rural area there might also be a limit on the available communication capabilities and computational power. Due to this, a promising candidate for control of irrigation networks is the structured optimal {LQ} controller with a distributed implementation studied in our previous paper \cite{heyden2021structured}.
We will in this section present a slight variation of that structured controller,
designed for a network model where the pools have the following first-order dynamics,
\begin{equation}\label{eq:first_delayed_des}
y_i[t+1] = y_i[t] + b_iu_i[t-\tau_i-\bar\tau] - c_i(u_{i-1}[t-\bar\tau] - d_i[t-\bar\tau]).
\end{equation}
The dynamics in \eqref{eq:first_delayed_des} is a first-order approximation of the third-order dynamics in \eqref{eq:third_order} when all the inputs and planned disturbances are low-pass filtered. The low-pass filter, which is used to  suppress the wave dynamics, is the source of  the additional delay $\bar\tau$.
The low-pass filter and the model in \eqref{eq:first_delayed_des} will be discussed further in the next section.

Before that, we will present the optimal controller for the first-order pool dynamics in \eqref{eq:first_delayed_des}. That is we study the following {LQ} control problem
\begin{equation}\label{eq:problem}\begin{aligned}
  \minimize_{y,u} \quad & \text{cost in } \eqref{eq:cost}\\
  \st \quad & \text{dynamics in \eqref{eq:first_delayed_des}}\\
  & y[0] \text{ and } d_i[t] \text{ given}.
\end{aligned}\end{equation}

The following Theorem shows that two algorithms can be used to calculate the necessary parameters for, and the implementation of, the optimal {LQ} controller for the problem in \eqref{eq:problem}. Both algorithms are implemented through a serial sweep using local communication and scalar computations. This means that the optimal {LQ} controller can be implemented in a distributed way.
\begin{theorem} \label{thm:struct}
Assume that $r_i = 0$ for $i\neq N$, $\rho_i =0$ for all $i$, and that $d_i[s] = 0$ for all $s>t+H$ for a fixed $H>0$.
Let $\sigma_i = \sum_{j=1}^{i-1}\tau_j$.
Then the minimizing $u_i[t]$ for the problem in \eqref{eq:problem}
 is given by running Algorithm 2 with the parameters from Algorithm 1.
\end{theorem}
\begin{proof}
The result is a minor extension of the results in \cite{heyden2021structured}. For completeness, the proof is given in the appendix.
\end{proof}


\begin{algorithm}[h!]
\DontPrintSemicolon
\SetNoFillComment
\SetKwInOut{Input}{input}\SetKwInOut{Output}{output}
\SetAlgoLined
\Input{$q_i$, $r$, $b_i$, $c_i$}
\Output{$\gamma_i$, $\hat{b}_i$, $g$}
\vspace{-0.25cm}
\nonl \hrulefill  \\
\textbf{send} $\gamma_1 = q_1$ and $\hat{b}_1 = b_1$ to upstream neighbor\\
\For (\tcp*[h]{Sweep through the pools}) {gate i = 2:N}   { 
$\hat b_i = {b_i}/{c_i}\cdot\hat b_{i-1}$\\
    $q_i = {c_i^2}/{\hat b_{i-1}^2}\cdot q_i$\\
    \vspace{2pt}
  $\gamma_i = \frac{\gamma_{i-1}q_i}{\gamma_{i-1} + q_i}$\\
  \textbf{send} $\gamma_i$ and $\hat{b}_i$ to upstream neighbor\\
} 
	$r = {r}/{\hat b_N^2}$ \tcp*[r]{Gate N}
	$X = -{\gamma_N}/{2}+\sqrt{\gamma_Nr+\frac{\gamma_N^2}{4}}$ \tcp*[r]{Gate N}
	$g = \frac{X}{X+\gamma_N}$ \tcp*[r]{Gate N}
\caption{Computation of control parameters.}
\label{alg:init}
\end{algorithm}

\begin{algorithm}[h!]
\DontPrintSemicolon
\SetNoFillComment
\SetKwInOut{Input}{input}\SetKwInOut{Output}{output}
\SetAlgoLined
\Input{$y_i[t]$, new $d_i[t]$, output from Algorithm 1}
\Output{$q_i$,$\gamma_i$,$\hat{b}_i$}
\vspace{-0.25cm}
\nonl \hrulefill  \\
\tcc{old $u_i[t-s]$ and $D_i[t+s]$ are kept in memory}
$y_i = \frac{\hat b_{i-1}}{c_i}y_i$ \tcp*[r]{Done in parallel for $i\geq 2$}
$d_1 = c_1d_1, \ \  d_i = b_{i-1} d_i \ i\geq 2$\\
\tcc{Update unchanged $D_i[t_0+\sigma_i]$}
\For(\tcp*[h]{Done in parallel, $\mathcal{O}(1)$}){gate i = N:2 }{ 
\textbf{Send} $D_{i-1}[t+\sigma_{i-1}+\tau_{i-1}] = D_{i}[t+\sigma_{i}] - d_{i}[t]$\\ \nonl \qquad downstream.\\
}
\tcp{Start sweep through graph}
$m_1[t] = y_1[t] + \sum_{s=1}^{\tau_i+\bar\tau}u_1[t-s]  \sum_{s=1}^{\bar\tau}d_1[t-s]+$\\ $\nonl \qquad \sum_{s=0}^{\tau_i}D_1[t+\sigma_i+s]$ \\
\textbf{send} $m_1[t]$ upstream \\
\For {gate i = 2:N} {
\tcc{ For $t+\sigma_i\leq s <t+\sigma_N+H$}
    \If{$d_i[s-\sigma_i]$ \emph{changed} or $D_{i-1}[s]$ \emph{received}}
    {\textbf{send} $D_i[s] = D_{i-1}[s] + d_i[s-\sigma_i] $ upstream
    }   
$p_i[t] = y_i[t] + \sum_{s=\tau_i}^{\tau_i+\bar\tau}u_i[t-s] $\\
\nonl $\qquad \qquad- \sum_{s=1}^{\bar\tau}u_{i-1}[t-s] + \sum_{s=0}^{\bar\tau}d_i[t-s]$\\

$ 	m_i[t] = m_{i-1}[t] + p_i[t]+\sum_{s=1}^{\tau_i-1}u_i[t-s]$\\
\nonl $\qquad \qquad + \sum_{s=1}^{\tau_i}D_i[t+\sigma_i+s]$\\
\textbf{send} $m_i[t]$ upstream \\
}
$u_{i-1} = (1-\gamma_i/q_i)p_i[t] -\gamma_i/q_i \cdot m_{i-1}[t], \quad 2\leq i \leq N$\\
$u_N[t] = -\frac{X}{r}\left[m_N + \sum_{s=\tau_N+1}^{H} D_N[t+\sigma_{N}+s]\prod_{j=2}^{d+1}g  \right]$\\
\textbf{send} $u_i$ downstream \tcp*[r]{Done in parallel}

$u_{i-1} = 1/\hat b_{i-1} \cdot u_{i-1}$ \tcp*[r]{For all gates} 
\caption{ Implementation of control law.}
\label{alg:impl}
\end{algorithm}

For both algorithms all measurements and calculations are made at the gates. Gate $i$ is at the end of pool $i$, and is responsible for deciding $u_{i-1}$.

Algorithm 1 can be used to calculate the parameters needed to implement the feedback law. The algorithm consists of a sweep through the graph.
On line 3-4 the parameters $b$ and $q_i$ are re-scaled, which corresponds to transforming the dynamics in \eqref{eq:first_delayed_des} to the form in \cite{heyden2021structured}. On line 5 the parameter $\gamma_i$ is calculated recursively. Finally, when the sweep is completed, the parameters needed to calculate the optimal outflow from the reservoir are calculated on line 8-10, including another scaling on line 8.

Algorithm 2 is used for the online implementation of the optimal controller.
The algorithm assumes that each gate stores its incoming and outgoing flow and the disturbance sums $D_i[s]$, defined as 
\[
D_i[t] = \sum_{j=1}^i d_j[t-\sigma_j].
\]
Line 1-2 is a change of variables.  On line 3-5 the $D_i[s]$ 
for which no new disturbances $d_i$ are announced are updated. Only one $D_i$ for each gate needs to be sent downstream, as the rest of the needed $D_i$ were already known in the gate form the previous time point. This can be done in parallel for all gates.

Next a serial sweep starts at the most downstream pool (pool one) and goes through the graph in the upstream direction.
The sweep accomplishes two things. 
Firstly, on lines 9-11 all $D_i$ for which a new disturbance $d_j$ was announced are updated. When the controller is initialized all non zero $D_i$ need to be updated this way. Secondly, $m_i[t]$ and $p_i[t]$ which are used for the calculation of $u_i$ are calculated on lines 12-14. The variable $p_i[t]$, which is the predicted level in pool $i$ at time $t+\bar\tau + 1$ when the outflow $u_{i-1} = 0$, is calculated on line 12. The calculation of $p$ only requires local and neighboring  information, where the incoming flow to pool $i$ from gate $i$ must be known. For the  calculation  of $m_i[t]$ on line 13, which is the total level in the first $i$ pools, only local information, $p_i[t]$, and the previous $m_{i-1}[t]$ is needed.
Finally $m_i[t]$ is sent upstream on line $14$.

After the sweep is completed all the inputs can be calculated on lines 16-17, relying only on $p_i$ and $m_{i-1}$ (and $D_N$ for $u_N$). The input $u_i[t]$ is then sent downstream to gate $i$ on line 18, as it is needed for the calculation of $p_i$ and $m_i$ in future time-steps.
Finally, all inputs are re-scaled on line 19. A sketch of the information flow for the implementation is found in Fig.~\ref{fig:impl}

\begin{figure}
    \centering
    \includegraphics{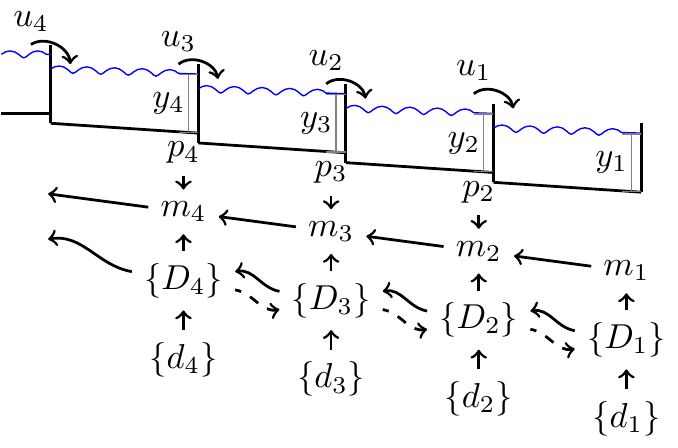}
    \caption{Illustration of the communication structure for the structured controller with 4 pools. The value for $m_i$ is calculated by a sweep through the graph, requiring the downstream $m_{i-1}$, the local $p_i$ and a local set of disturbances $D_i$ (line 12-14 in Algorithm 2). The disturbances $D$ can be calculated in two ways. If any new underlying $d_j$ is announced, then the corresponding $D_i$ must be calculated through a similar sweep to $m$, going through the graph upstream (line 9-11 in Algorithm 2), illustrated by solid arrows . However, if there are no new planned disturbance $d_j$, then the aggregate disturbances $D_i$ can be updated from the upstream gate (line 3-4 in Algorithm 2), illustrated by dashed arrows. Note that for the outflow from the reservoir $u_4$, $m_4$ and $D_4$ will be used and hence they are sent to that gate as indicated by the arrows. }
    \label{fig:impl}
\end{figure}

\section{Applying the Structured Controller to an Irrigation Network}
\label{sec:applying}
In this section, we will go through the steps taken to apply the controller presented in Section \ref{sec:struct_contr} to the simulation models with third-order pool dynamics (presented in Section \ref{sec:problem_description}). While the previous section had a strong theoretical motivation, this section will be more practical. The steps taken here are certainly not the only way to apply the structured controller just presented to the irrigation network model with third order pool dynamics, but constitute a simple and transparent approach.

\subsection{Low-pass Filter}
The third-order system has a poorly damped node, which introduces two problems. Firstly, one wants to avoid introducing waves into the pools. And secondly, the structured {LQ}-controller must be designed using a first-order model, which can not describe the frequency peak.

One alternative to remedy both issues is to design an inner controller at the gate which takes a flow reference and then controls the flow. It should be designed so that the transfer function from the flow reference to the level in the pool would be close to first-order. This would require a detailed model of the pools on both sides of the gate.

We instead choose to add a low-pass filter to each input and each planned disturbance. Filtering the disturbance is natural since waves should be avoided both in the pools and in the off-takes to the farmers. However, additional consideration might need to be taken to make sure that the farmers get the amount of water that they ordered and that the delivery time is not delayed too much by the low pass filter. This could be accomplished by, for example, modifying the farmer's order before applying the low-pass filter. 

The low pass filter at each gate must be designed based on its two neighboring pools so that no waves are induced in either pool.
For simplicity, we use the same low pass filter for all gates, which then must suppress the wave dynamics in both pools.
A Butterworth filter is used for the low-pass filter, as it has minimal effect on the pass-band. The Matlab command \texttt{butter} is used for the design and the final design is a  third-order filter with a cut-off frequency  $3\cdot 10^{-3}$ rad/sec. The resulting bode magnitude plot before and after the low-pass filtering can be found in Fig.~\ref{fig:bodemag}. The third-order models are used both in the design and the evaluation of the low-pass filter. However, if detailed models were not available, it would still be possible to design the low-pass filter based only on knowledge of the dominant wave frequency and evaluate it using open-loop tests in the canals.

\begin{figure}
\centering
\includegraphics[width = 0.9\textwidth]{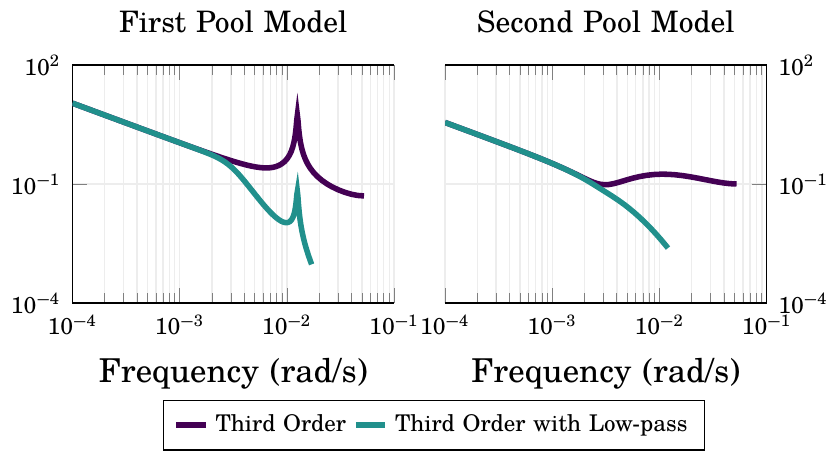}
\caption{The bode magnitude plot for the transfer function from inflow $u_i$ to level $y_i$ for the two sets of parameters in Table \ref{tab:parameters} for \eqref{eq:third_order}, with and without low-pass filter. It can be seen that for low frequencies the third-order pool model can be well described by a first-order system, but for higher frequencies, there is a resonance peak. The low-pass filter manages to suppress this peak.}
\label{fig:bodemag}
\end{figure}

\subsection{First-order approximation}
First-order models in the form of
\begin{equation}\label{eq:first_non_delayed}
	y_i[t+1] = y_i[t] + b_iu_i[t-\tau_i] - c_i(u_{i-1}[t] - d_i[t]),
\end{equation}
 where $b_i>0$ and $c_i>0$, have been shown to describe the water level in a pool well on slow timescales \cite{cantoni2007control}. Just as in the third-order model, in the above $u_i = (y_i-p_i)^{3/2}$, and $y_i$ denotes the water level in the \emph{i}th pool. Parameters for a suitable first-order description of the same two pools from the Haughton main river were given in \cite{weyer2003lq}. However, upon closer examination there was a large difference between the first and third-order model in terms of their DC gains for the inflow into the second pool. To counteract this, we modified the first-order model, where $b_{i,1}$ was increased by a factor of $1.5$. In practice a more principled approach should be used to construct a suitable reduced order model, however it is reassuring that working in this ad-hoc manner still resulted in a good enough model for conducting synthesis.
 
To handle the addition of the low-pass filter we propose a simple update to \eqref{eq:first_non_delayed} as already given in \eqref{eq:first_delayed_des}
\begin{equation*}
	y_i[t+1] = y_i[t]+b_i u_i[t-\tau_i-\bar\tau] - c_i(u_{i-1}[t-\bar\tau]-d_i[t-\bar\tau]).
\end{equation*}
The additional delay $\bar\tau$, which is the same for all pools, can intuitively be motivated as an approximation of the effect of the low pass filter. We also let $\tau_i$ be different from the ones in \cite{ooi2005physical}, as it was noted that this had a positive effect on the performance. The parameters $b_i$ and $c_i$ are unchanged.

The parameters $\tau_i$ and $\bar\tau$ are chosen as follows.
First the optimal $\bar\tau$  is found by simulating the response for both pools to an outflow $u_{i-1}$  corresponding to a constant positive input, followed by a constant zero input, followed by a negative input.
That is
\begin{equation}\label{eq:sys_id}
	u[t] = \begin{cases}
	{\color{white}-}1 \quad &t<t_1\\
	{\color{white}-}0 \quad &t_1 \leq t < t_2 \\
	-1 \quad &t_2 \leq t < t_3\\
	{\color{white}-}0 \quad & t\geq t_3.
	\end{cases}
\end{equation}
The idea is that this describes when a pool is emptied and then filled. A similar open-loop experiment could easily be conducted in an irrigation network.
   The value for $\bar\tau$ that minimizes the least square error (normalized for each pool)  is chosen. This is an integer optimization problem, but the number of reasonable values are limited so we can expect to find the optimal value.
   The resulting value for $\bar\tau$ is $10$. Next, the optimal $\tau_i$ for each pool is found by minimizing the least square error when the inflow $u_i$ is as in \eqref{eq:sys_id}. The resulting value for the first pool model is $\tau_i = 2 $ and for the second pool model $\tau_i = 15$. The resulting system responses when the inflow is zero and the outflow is as in \eqref{eq:sys_id} are plotted in Fig.~\ref{fig:sys_id}, where it can be seen that the first-order system gives a good approximation of the low-pass filtered third-order system.
   
   \begin{figure}
    \centering
    \includegraphics[width = 0.9\textwidth]{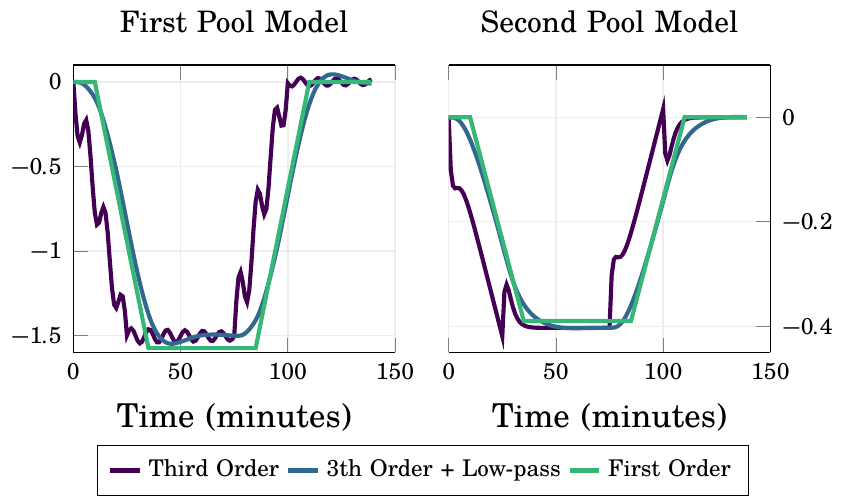}
    \caption{Time response from the outflow $(u_{i-1})$  for the third-order model in \eqref{eq:third_order}, third-order model with low pass, and first-order model with additional delay in \eqref{eq:first_delayed_des}. It can be seen that the low pass filter suppresses most oscillations, and that the first-order pool model captures the behavior of the third-order pool model with a low-pass filter well. }
    \label{fig:sys_id}
\end{figure}

\subsection{Kalman Filter}
The first-order approximation describes the behavior on slow time scales of the third-order model with a low-pass filter. However, there are still some differences. For example, the step response for the first-order model starts slower but finishes faster. These differences can be handled by introducing a Kalman filter, so that in the short term the controller trusts the first-order model, but in the long term it still utilizes the measurements from the third-order model.

For the Kalman filter design we consider the same dynamics used in the controller design, but with added (unknown) state disturbance $v_i[t]$ and measurement disturbance $w_i[t]$,
\[\begin{aligned}
	x_i[t+1] &= x_i[t] + b_iu_i[t-\tau_i-\bar{\tau}] \\
	& \qquad \quad- c_i(u_{i-1}[t-\bar{\tau}]-  d_i[t-\bar{\tau}]) + w_i[t] \\
	y_i[t] & = x_i[t] + v_i[t].
\end{aligned}\]
Changing the relationship between the modeled variance of $w_i$ and $v_i$ allows balancing how much the Kalman filter trusts the measurements compared to the first-order model.

The Kalman filter is updated using the following scalar dynamics which can be implemented locally at each gate,
\[\begin{aligned}
	\hat{y}_{t|t} &  = \hat{y}_{t|t-1} + L(y_{t} - \hat{y}_{t|t-1}) \\
	\hat{y}_{t+1|t}  & = \hat{y}_{t|t} + b_iu_i[t-\tau_i-\bar{\tau}] -c_i( u_{i-1}[t-\bar{\tau}] - d_i[t-\bar{\tau}]).
\end{aligned}\]
In the above, $L$ is the solution to the scalar Riccati equation,
\[
	L = L - L^2/(L+R_2) + R_1,
\]
where $R_1$ is the variance of $w_i$ and $R_2$ is the variance of $v_i$.
 For the simulations we use $R_1 = 1$ and $R_2 = 100$.
The a priori estimate $x_{t|t-1}$ is used in the calculation of the inputs at time $t$. This gives a minute of time for propagating information through the string graph.

\section{Comparison Controllers}
We design two additional controllers to use for comparisons with the structured controller. Firstly, we design a {LQ} controller using the third-order pool model in \eqref{eq:third_order} to get the best possible performance in terms of the performance criterion in \eqref{eq:cost}. Secondly, we design a simple P controller that will give a baseline in terms of easily achievable performance.

To get a fair comparison, the disturbance will be low pass filtered for these controllers as well.
 Furthermore, as the structured controller does not have integral action but instead relies on feed-forward to reject load disturbances, we let the standard {LQ} controller and P controller also use feed-forward and have no integral action.

\subsection{Third-Order {LQ}}
To get a baseline of the best possible performance we consider an {LQ} controller synthesized directly on the third-order dynamics. This controller is not meant to be implementable in practice so we let the controller have access to full state information.

Consider a state space representation for the transfer function in \eqref{eq:third_order} on 
the form
\[\begin{aligned}
	x_i[t+1] &= A_ix_i[t] + B_i(1)u_i[t] + B_i(2)u_{i-1}[t]\\
	y_i[t] & = C_ix_i[t].
\end{aligned}\]
Then using the dynamics $A = \text{diag}(A_1,A_2,\dots,A_N)$, $C = \text{diag}(C_1,C_2,\dots, C_N)$,
\[\setlength\arraycolsep{2pt}
	B = \begin{bmatrix}
		B_1(1) &0 \\
		B_2(2) &B_2(1) & 0  \\
		 
		 \ddots &\ddots & \ddots & \ddots \\
		  & 0 & B_N(2) & B_N(1)
	\end{bmatrix},\
	v[t] = \begin{bmatrix}
		B_1(2)d_1[t]\\
		B_2(2)d_2[t]\\
		\vdots\\
		B_N(2)d_N[t]
	\end{bmatrix},
\]
the dynamics of the water levels $y_i[t]$ for a network with $N$ pools can be described by
\[\begin{aligned}
	x[t+1] &= Ax[t] + Bu[t] + v[t] \\
	y[t] & = C x[t].
\end{aligned}\]
Let $Q = \text{diag}(q_1C_1^TC_1,\ q_2C_2^TC_2,\ \dots,\ q_NC_N^TC_N)$, then the cost due to the pool levels can be expressed as $\sum_{i=1}^N q_iy_i[t]^2 =x[t]^TQx[t]$.
Now, let $S$ be the solution to the Riccati equation
\[
	S =A^TSA- A^TSB(B^TSB+R)^{-1}B^TSA + Q,
\]
and define
\begin{equation*}
\begin{aligned}
	K & = - (B^TSB+R)^{-1}B^TSA \\
	K_d & = -(B^TSB+R)^{-1}B^T \\
	\Pi[t] & = (A+BK)^T\Pi[t+1] + Sv[t], \quad \Pi[H+1] = 0.
\end{aligned}\end{equation*}
Then the optimal input $u[t]$ is given by 
\[
	u[t] = Kx[t] + K_d\Pi[t].
\]
A derivation of the optimal feed-forward for the known disturbance can be found in the appendix of this paper.

To allow for a penalty on the change in input $(u_i[t]-u_i[t-1])^2$ we introduced new states, corresponding to $u_i[t-1]$ and $(u_i[t]-u_i[t-1])^2$.
Additional states could be introduced to further improve the performance of the {LQ} controller, such as penalizing a high pass filtered version of the output to reduce the oscillations in the system, see for example \cite{weyer2008control}. As this {LQ} controller is only used to get the maximum performance, we consider only the aspects captured by the performance measure.

\subsection{P-Controller}
For the P-controller we consider a configuration where the controller at gate $i$ is designed to control the water level at the end of pool $i-1$, which is the level just before the downstream gate $i-1$. This setup is often called distant downstream control \cite{weyer2002decentralised}. The low-pass filter that was used to filter the inputs for the structured controller is also used for the P-controller.
We use feed-forward both on the outflow from the downstream gate $i-1$ and on the off-take at the downstream gate. Thus the controller is in the form of
\[
	u_i[t] = -k_i y_i[t] + k_{ff}\frac{c_i}{b_i}(u_{i-1}[t-1] - d_i[t+\tau_i]).
\]
The fraction $c_i/b_i$ is used to account for the different coefficients in the inflow and outflow. The feed-forward on the downstream input $u_{i-1}[t-1]$ is delayed as otherwise $u_i[t]$ would depend on all $u_j[t]$ for $j<i$.

We use the following values for the controller parameters,
\begin{equation}
\label{eq:p_params}
k_i = \frac{\pi}{2(\tau_i+\bar\tau)b_i}\frac{1}{4},\quad k_{ff} = 1.
\end{equation}

The choice of $k_i$ was partially found by hand-tuning, but can also be theoretically motivated. For the  design of the P-controller the outflow from the downstream gate can be modeled as a  disturbance. Using the model in \eqref{eq:first_delayed_des} gives the following continuous time dynamics
\[
	\dot{y}_i = b_i u_i(t-\tau_i-\bar\tau) + d_i(t) \Rightarrow G(s) = \frac{b_i}{s}e^{-(\tau_i+\bar\tau) s}.
\]
Ignoring the feed-forward, the controller is in the form of $u_i(t) = -k_i y_i(t)$, which gives the loop transfer function
\[
	\frac{k_ib_i}{s}e^{-(\tau_i+\bar\tau)s}.
\]
Picking $k_i$ as in \eqref{eq:p_params}, the time responses for all the pools will have the same shape, but with different time constants.
Considering the gain margin
\[
	\frac{\pi}{2(\tau_i+\bar\tau)b_ik_i}
\]
and the phase margin
\[
	\frac{\pi}{2} - (\tau_i+\bar\tau)b_ik_i
\]
shows that the choice of $k_i$ gives a gain margin of $4$ and a phase margin of $67.5$ degrees.

\section{Simulations}
In this section we use the two networks discussed in Section \ref{sec:problem_description} to compare the performance of the three different controllers.
In the first part  we consider cost functions that satisfies the assumption for the structured controller, that is $r_i = 0,\ i\neq N$ and $\rho_i=0$.  We explore both the time response for the different controllers, and study how they scale with the size of the network. Next we explore the limitations for the structured controller by comparing how well one can balance the deviations in inputs and in the levels. All code used for the simulation is available on GitHub\footnote{\url{https://github.com/Martin-Heyden/ECC-irrigation-network}}.

\begin{sidewaysfigure}
\centering
\includegraphics[]{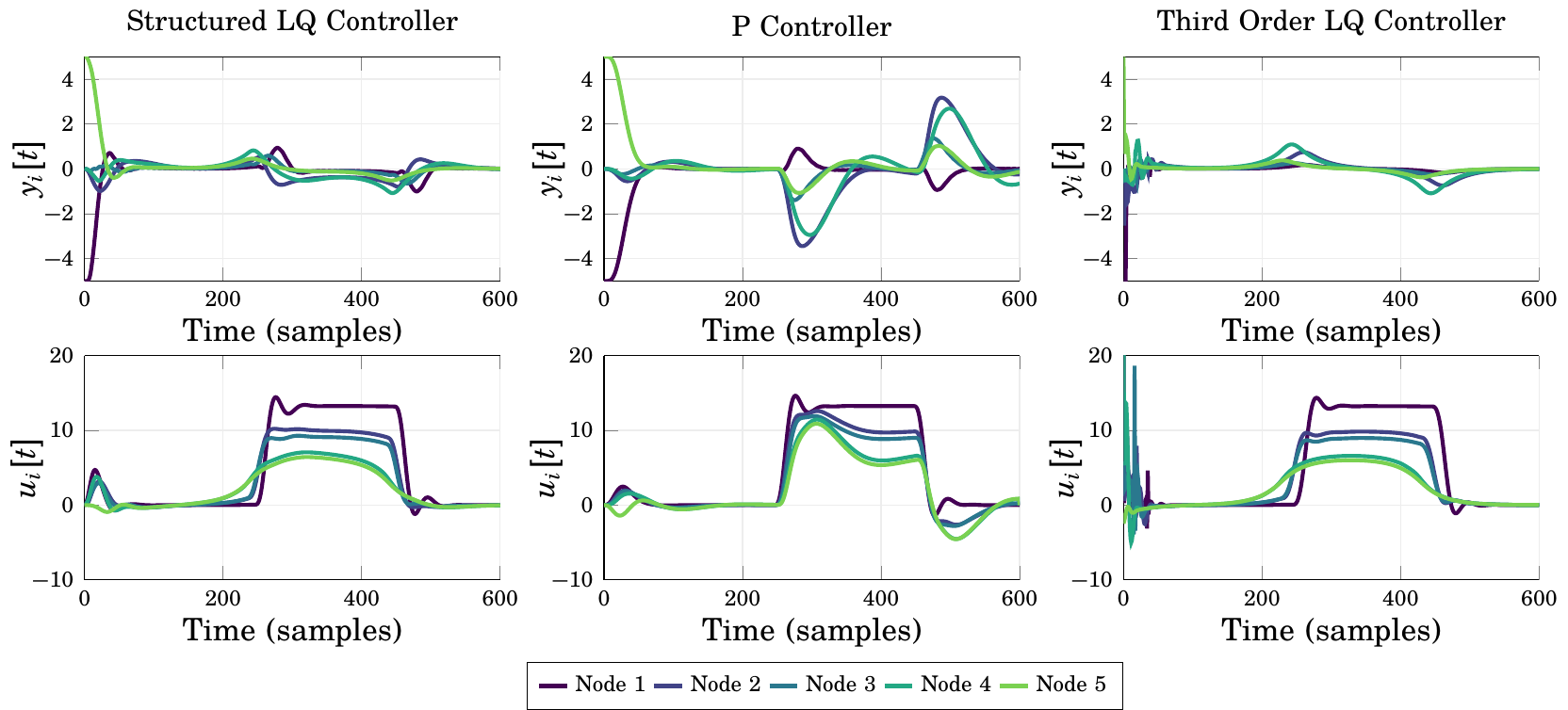}
\caption{Illustration of the time response for the 
different controllers. The systems starts with initial condition $[-5,0,0,0,5]$ corresponding to a step change. Between time 250 and 450 there is a disturbance in pool one with a discharge rate that gives a change of one unit per minute to the level.}
\label{fig:time_resp}
\end{sidewaysfigure}

We use the cost function parameters $q_i = 1$, $r_N = 0.3$, $r_i = 0,\ i\neq N$ and $\rho_i=0$. In Fig. \ref{fig:time_resp} the time responses for the three different controllers are depicted. Canal one, three, and five are modeled as the first pool and canal two and four are modeled as the second pool. The initial condition is $[5,0,0,0,-5]$, corresponding to a change in set-point resulting in water needing to be moved through the graph. Then there is a disturbance in pool one between time 250 and 450, corresponding to a change in level of 1 unit/minute. It can be seen that the third-order {LQ}-controller is very aggressive for the step response and this step response would neither be wanted, nor implementable at the gates.

Next we consider how the performance of the different controllers scales with the size of the network. From now on, all pools have the dynamics in the first pool model. This allows us to clearly see the effect of the  varied parameter. Also, to limit the effect of the design decision for the P-controller, we ran a set of different controllers with $k_i$ as a factor of [0.25, 0.5, 1, 1.5, 2] of the nominal value, and picked the best performance for each configuration.
The left graph in Fig.~\ref{fig:dist_perf} depicts how the change in the number of pools affect the cost when the disturbance is kept in pool $N-1$, which is the second pool counting from the reservoir. For the right graph, the number of pools in the network is fixed to 10, and the pool which the disturbance acts upon is varied. 
For both cases the disturbance is acting between time $200$ and $400$. We can see that the two {LQ}-controllers improves performance slightly when the graph size increases, while the P-controller does not utilize the additional pools. 
A bigger difference is seen when the disturbance pool is varied. Here it can be noted that there is an increase in performance for all the controllers when the disturbance pool is far away from the reservoir. For the P-controller this is partly due to the controller only using the pools upstream of the disturbance. In general, the reason that the performance is increased when the disturbance is further downstream  could be that it is more efficient when the transportation from the reservoir and the other pools are all in the same direction.
We also note that the performance of the two {LQ}-controllers are almost identical for both cases. This is likely due to the fact that the disturbances are low-pass filtered, and thus the low-pass filtering of the inputs for the structured controllers does not limit the performance much.

\begin{figure}
    \centering
    \includegraphics[width = 0.9\textwidth]{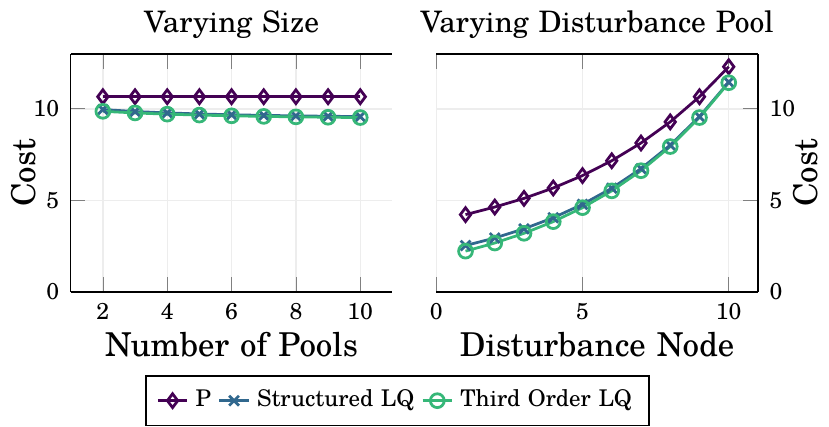}
    \caption{Comparison of the performance for the different controllers when there is a planned disturbance in the network. In the left figure the disturbance is always in pool $N-1$ and $N$ is varying. In the right figure $N$ is fixed to 10 and the pool with the disturbance is varied.}
    \label{fig:dist_perf}
\end{figure}

Indeed, in Fig.~\ref{fig:step_perf} we consider the performance when there is a change in set-points, requiring water to be moved from the $N$'th pool (the pool after the reservoir) to the first pool (the most downstream one). Unsurprisingly, it can be seen that the third-order {LQ} controller outperforms the structured controller, as it can directly cancel out the waves.
On the other hand, the time response in Fig. \ref{fig:time_resp} indicated that the third-order {LQ} controller needs to be made less aggressive, and the performance of the third-order {LQ} controller can most likely not be reached with a controller suitable for implementation.
The difference between the structured controller and the P-controller is bigger here than for the disturbance rejection.

\begin{figure}
    \centering
    \includegraphics[width = 0.9\textwidth]{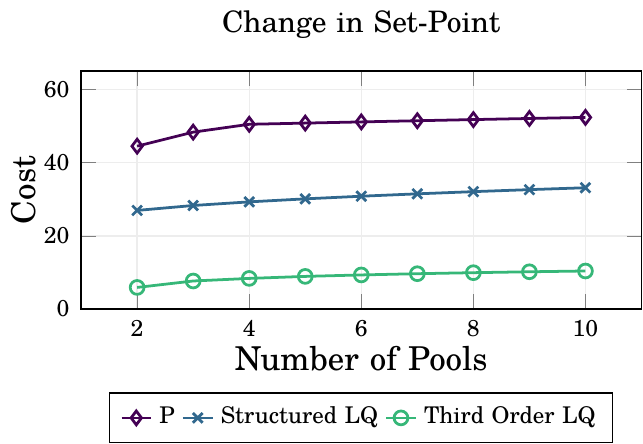}
    \caption{The performance for the different controller for non zero initial conditions, corresponding to a change in set-point. The initial conditions are $y_1 = -1$, $y_N = 1$ and $y_i=0$ for $2 \leq i \leq N-1$.}
    \label{fig:step_perf}
\end{figure}

Finally, we consider how well the trade-off between input deviations and state deviations can be handled by the structured controller. In Fig.~\ref{fig:inp_comp} we have plotted $\sum y_i[t]^2$ on the x-axis and $\sum u_i[t]$ and $\sum (u_i[t]-u_i[t-1])^2$ respectively on the y-axis for different design parameters. For the structured {LQ} controller $r_N$ is varied and for the third-order {LQ} controller  $r_i$ and $\rho_i$ respectively are varied. The simulations are carried out on a ten pool network with a disturbance in pool five between time 200 and 400.
For the trade-off between the quadratic deviations in the input and in the states, the structured controller allows through the parameter $r_N$ to hold up quite well to the third-order {LQ} controller. However, it can be seen that the difference is larger for lower input deviations, which is to be expected.
When it comes to minimizing the square of change in input, $(u_i[t]-u_i[t-1])^2$, the structured {LQ} controller have only a limited ability to influence the trade off though the parameter $r_N$. Consequently, the trade off becomes quickly worse than for the standard third-order {LQ} controller. However, if we consider the time response in Fig.~\ref{fig:time_resp} the input variations look quite timid, with it being almost constant during the disturbance. If one wanted to reduce the input changes further, one could consider adding an additional local controller to the low-pass filter that minimizes the input changes.

\begin{figure}
    \centering
    \includegraphics[width = 0.9\textwidth]{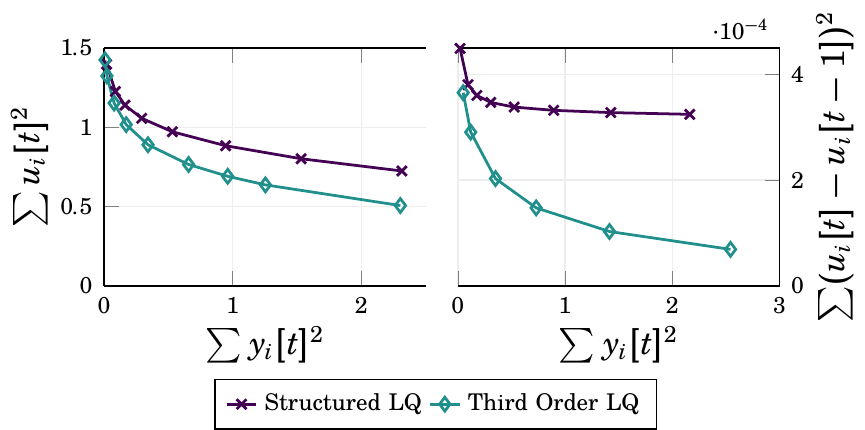}
    \caption{Comparison for how well the two different {LQ} controllers can handle the trade off between level deviations and input deviations. In the left figure each data point shows the square of input deviations and levels deviations for a choice of design parameters. In the right figure the square of the levels and change in input $(u_i[t]-u_i[t-1])^2$ is plotted for different design parameters. It can be seen that the structured controller can do a fairly good job of handling the trade off between levels and input deviation, while the trade off between levels deviations and change in input is worse. }
    \label{fig:inp_comp}
\end{figure}

\appendix
\section{Appendix}

\subsection{Proof of Theorem 1}
In this section Theorem 1 will be proven. We start with flows between two pools, that is $u_i,\ i <N$, and then find the optimal flow $u_N$ from the reservoir. We remind ourselves of the following definitions, which will be used in the proof:
\[
  \sigma_i = \sum_{j=1}^{i-1}\tau_i,\quad D_i[t] = \sum_{j=1}^i d_j[t-\sigma_j].
\]
The proof will rely on results presented in the extended version of \cite{heyden2021structured}\footnote{That version can be found within the paper}. Note that in that paper the notation for $u_N$ is $v_N$. Furthermore, we call the level in each node $y$ instead of $z$ in this paper.

\subsubsection{Optimal Internal Flows.}
We will derive the optimal controller for the following dynamics
\begin{equation}\label{eq:dyn_proof}
  y_i[t+1] = y_i[t]+u_i[t-\tau_{i}-\bar\tau] - u_{i-1}[t-\bar\tau]+d_i[t-\bar\tau].
\end{equation}
 When that is done, we will present a change of variables that transforms the dynamics to the model used for control synthesis in \eqref{eq:first_delayed_des}.

To get back to the dynamics studied in \cite{heyden2021structured} we apply the following change of variables.
Let $\nu_i[t] = u_i[t-\bar\tau]$ for $i\leq N-1$, $\nu_N[t] = u_N[t-\tau_N-\bar\tau]$, $\delta_i[t] = d_i[t-\bar\tau]$, $\Delta_i[t] = D_i[t-\bar\tau]$, and finally
\[
  \xi_k[t] = \sum_{i=1}^k\Big( y_i[t] + \sum_{s=1}^{\tau_i}\nu_i[t-s]\Big).
\]
In terms of these variables the dynamics in \eqref{eq:dyn_proof} are given by
\begin{equation*}
\begin{aligned}
  y_1[t+1] & = y_1[t]+\nu_1[t-\tau_1] +\delta_1[t]\\
  y_i[t+1] & = y_i[t]+\nu_i[t-\tau_i] - \nu_{i-1}[t]+\delta_i[t]\\
  y_N[t+1] & = y_N[t] + \nu_N[t] - \nu_{N-1}[t] + \delta_N[t],
  \end{aligned}
\end{equation*}
which are the the dynamics studied in \cite{heyden2021structured}, but with production only in the top node.
Lemma 1-iii from \cite{heyden2021structured} holds for any production, and we can always change what time is defined as zero, which gives that for $k<N$ (it holds that $v_k[t] = 0$ and $\bar{V}_i[t] = D_i[t],\ i<N$)
\begin{multline*}
    \nu_{k-1}[t] = (1-\frac{\gamma_{k}}{q_k})(y_{k}[t] + \nu_{k}[t-\tau_{k}])  +\delta_k[t]\\
  -\frac{\gamma_k}{q_k}\Big(\xi_{k-1}[t]+{\Delta}_k[t+\sigma_k] + \sum_{i=1}^{k-1} \sum_{d=0}^{\tau_i-1}{\Delta}_i[t+\sigma_i+d]\Big).
\end{multline*}
For the outflow of node N, which has production, it holds that  $\bar{V}_N = D_N[t]+ v_N[t]$. And thus $\nu_{N-1}$ is given by
\begin{multline*}
    \nu_{N-1}[t] = (1-\frac{\gamma_{N}}{q_N})y_{N}[t]  +\delta_N[t] + \nu_N[t]\\
  -\frac{\gamma_N}{q_N}\Big(\xi_{N-1}[t] +\nu_N[t] +\Delta_N[t+\sigma_N]+  \sum_{i=1}^{N-1} \sum_{d=0}^{\tau_N-1}{\Delta}_N[t+\sigma_N+d]\Big).
\end{multline*}
 Rewriting either expression in terms of the original variables and shifting the time variable by $\bar\tau$ gives for $k<N$
 \begin{multline}\label{eq:u_non_causal}
 u_{k-1}[t] = (1-\frac{\gamma_k}{q_k})(y_k[t+\bar\tau] + u_k[t-\tau_k]) + d_k[t] \\
 -\frac{\gamma_k}{q_k}\left[\sum_{i=1}^{k-1}\Big(y_i[t+\bar\tau] +\sum_{s=1}^{\tau_i}u_i[t-s] \Big) +  D_k[t+\sigma_k] + \sum_{i=1}^{k-1}\sum_{s=0}^{\tau_i-1}D_i[t+\sigma_i+s]\right].
 \end{multline}
This expression can not be used for implementation, as $y_i[t+\bar\tau]$ is not known at time $t$.
 This problem is however easily solved by using the dynamics, which gives that
 \[ y_i[t+\bar{\tau}] = y_i[t] + \sum_{s = 1}^{\bar\tau} \Big(-u_{i-1}[t- s] + d_i[t-s] + u_{i}[t-\tau_i - s] \Big).\]
Collecting all terms  in \eqref{eq:u_non_causal} gives for $u_i$, $1\leq i<k-1$,
\[
  -\frac{\gamma_k}{q_k}\Big(\sum_{s=1}^{\tau_i} u_i[t-s] + \sum_{s=\tau_i+1}^{\tau_i+\bar\tau} u_i[t-s] - \sum_{s=1}^{\bar\tau}u_i[t-s] \Big) = -\frac{\gamma_k}{q_k}\sum_{s=\bar\tau+1}^{\tau_i+\bar\tau} u_i[t-s],\]
for $u_{k-1}$
\begin{multline*}
  (1-\frac{\gamma_k}{q_k})(-\sum_{s=1}^{\bar\tau}u_{k-1}[t-s])-\frac{\gamma_k}{q_k}\Big(\sum_{s=1}^{\tau_{k-1}} u_{k-1}[t-s] + \sum_{s=\tau_{k-1}+1}^{\tau_{k-1}+\bar\tau} u_{k-1}[t-s]\Big)\\
  =-\frac{\gamma_k}{q_k}\sum_{s=\bar\tau+1}^{\tau_{k-1}+\bar\tau} u_{k-1}[t-s]-\sum_{s=1}^{\bar\tau}u_{k-1}[t-s],
\end{multline*}
and finally for $u_k$
\[
  (1-\frac{\gamma_k}{q_k})\Big(u_k[t-\tau_k] + \sum_{s=\tau_k+1}^{\bar\tau+\tau_k}u_k[t-s]\Big) =(1-\frac{\gamma_k}{q_k}) \sum_{s=\tau_k}^{\bar\tau+\tau_k}u_k[t-s].
\]
This gives that
\begin{multline*}
 u_{k-1}[t] = (1-\frac{\gamma_k}{q_k})(y_k[t] + \sum_{s=\tau_k}^{\tau_k+\bar\tau} u_k[t-s] + \sum_{s = 1}^{\bar\tau}d_k[t-s]) + d_k[t]   - \sum_{s=1}^{\bar\tau}u_{k-1}[t-s]\\
 -\frac{\gamma_k}{q_k}(\sum_{i = 1}^{k-1}\Big( y_i[t] + \sum_{s=\bar\tau+1}^{\tau_i+\bar\tau} u_i[t-s]  +\sum_{s = 1}^{\bar\tau}d_i[t-s]\Big) + \sum_{i=1}^{k-1}\sum_{s=0}^{\tau_i-1}D_i[t+\sigma_i+s] + D_k[t+\sigma_k]),
 \end{multline*}
which is equal to
\[\begin{aligned}
   u_{k-1}[t] = (1-\frac{\gamma_k}{q_k})\Big[y_k[t] + \sum_{s=\tau_k}^{\tau_k+\bar\tau} u_k[t-s]- \sum_{s=1}^{\bar\tau}u_{k-1}[t-s] + \sum_{s = 0}^{\bar\tau}d_k[t-s]\Big]&   \\
 -\frac{\gamma_k}{q_k}\Big[\sum_{i = 1}^{k-1}\Big( y_i[t] + \sum_{s=\bar\tau+1}^{\tau_i+\bar\tau} u_i[t-s]  +\sum_{s = 1}^{\bar\tau}d_i[t-s]\Big) +  \sum_{s=1}^{\bar\tau}u_{k-1}[t-s]& \\+  \sum_{i=1}^{k-1}\sum_{d=0}^{\tau_i-1}D_i[t+\sigma_i+d] + (D_k[t+\sigma_k]-d_k[t]\Big]&.
\end{aligned}\]
 Now, let
 \[
  p_k[t] = y_k[t] + \sum_{s=\tau_k}^{\tau_k+\bar\tau}u_k[t-s] - \sum_{s=1}^{\bar\tau}u_{k-1}[t-s] + \sum_{s=0}^{\bar\tau}d_k[t-s]
 \]
 and
 \begin{multline*}
  m_k[t] =  \sum_{i=1}^k\left[y_i[t] +\sum_{s=\bar\tau+1}^{\tau_i+\bar\tau}u_i[t-s] + \sum_{s=1}^{\bar\tau} d_i[t-s]+ \sum_{s=0}^{\tau_i-1}D_i[t+\sigma_i+s] \right]  \\
  +  \sum_{s=1}^{\bar\tau}u_k[t-s] + D_{k}(t+\sigma_{k+1}).
 \end{multline*}
 Since $D_k[t+\sigma_k]-d_k[t] = D_{k-1}[t+\sigma_k]$ it then holds that
 \[
  u_{k-1}[t] = (1-\frac{\gamma_k}{q_k})p_k[t] -\frac{\gamma_k}{q_k}m_{k-1}[t].
 \]
$m_k[t]$ can be calculated recursively as follows:
 \[
  m_k[t] = m_{k-1}[t] + p_k[t]+\sum_{s=1}^{\tau_k-1}u_k[t-s] + \sum_{s=1}^{\tau_k}D_k[t+\sigma_k+s]
 \]
 where it is used that $d_k[t]+ D_{k-1}[t+\sigma_{k-1}+\tau_k] = D_k[t+\sigma_k]$.

 \subsubsection{Optimal Production.}
 The steps in Lemma 3 in \cite{heyden2021structured}, can be carried out with $\rho_N = r$ to find the optimal  $\nu_N$ (that is $v_N$ in the lemma).  Equation 16 in \cite{heyden2021structured} will then give that
 \[
  \nu_N[t] =   -\frac{X}{r} \left(\xi_{N-1}[t] +\sum_{i=1}^{N-1}\sum_{s = \sigma_i}^{\sigma_{i+1}-1}\Delta_i[t+s] +  \mu_N[t]\right).
 \]
 Where $X$ is given by
 \[
  X =  -\frac{\gamma_N}{2}+\sqrt{\gamma_Nr+\frac{\gamma_N^2}{4}},
 \]
  and $\mu_N[t]$ is given by
\[
  \mu_N[t] = y_N[t] + \sum_{s=0}^{H} \Delta_N[t+\sigma_{N}+s]\prod_{j=2}^{s+1}g,
\]
where $g = X/(X+\gamma_N)$ and the product over an empty set is defined to be $1$.
  Note that in \cite{heyden2021structured} all $X_N(i)$  are the same, as otherwise $X_N(H+2)+\gamma_N$ would not be a solution to the Riccati equation, and that $\tau_N$ was defined as $H+1$ in \cite{heyden2021structured} for notational convenience.

Going back to the original variables and shifting the time variable by $\tau_n+\bar\tau$ gives
\begin{multline}\label{eq:u_N_proof}
  u_N[t] =   -\frac{X}{r} \Big[\sum_{i=1}^{N-1}\Big(y_i[t+\tau_N+\bar{\tau}]+\sum_{s=1}^{\tau_i}u_i[t+\tau_N-s] +
  \sum_{s = \sigma_i}^{\sigma_{i+1}-1}D_i[t+s + \tau_N]\Big) \\
   +  y_N[t+\tau_N+\bar\tau] + \sum_{s=0}^{H} D_N[t+\sigma_{N}+\tau_N+s]\prod_{j=2}^{s+1}g\Big].
\end{multline}
Using the dynamics to rewrite $y_i[t+\tau_N+\bar\tau]$ gives
\[
  y_i[t+\tau_N+\bar{\tau}]  = y_i[t] + \sum_{s =\tau_i+1-\tau_N}^{\tau_i+\bar\tau}u_i[t-s] -\sum_{s=1-\tau_N}^{\bar\tau}u_{i-1}[t-s]+\sum_{s=1-\tau_N}^{\bar\tau}d_{i}[t-s].
\]
One can note that all terms in the RHS will not be known at time $t$. However, the issue solves itself as follows.
Collecting all terms containing $u_i$ $i\leq N$ in \eqref{eq:u_N_proof} gives
\[
  \sum_{s=1-\tau_N }^{\tau_i-\tau_N}u_i[t-s]+ \sum_{s =\tau_i+1-\tau_N}^{\tau_i+\bar\tau}u_i[t-s]-\sum_{s=1-\tau_N}^{\bar\tau}u_{i}[t-s] = \sum_{s=\bar\tau+1}^{\tau_i+\bar\tau} u_i[t-s].
\]
And all terms including $u_N$ gives
\[
   \sum_{s =-\tau_N+\tau_N+1}^{\tau_N+\bar\tau}u_N[t-s] =  \sum_{s =1}^{\tau_N+\bar\tau}u_N[t-s].
\]
And we note that both sums are now quantities known at time $t$.

All the disturbances $d_i[t]$ for $t\geq 0$ are given by
\begin{equation}\label{eq:u_n_d}
  \sum_{i=1}^N \sum_{s=0}^{\tau_N-1}d_i[t+s] + \sum_{i=1}^{N-1}\sum_{s = \sigma_i}^{\sigma_{i+1}-1}D_i[t+s + \tau_N] + \sum_{s=0}^{H} D_N[t+\sigma_{N}+\tau_N+s]\prod_{j=2}^{d+1}g. 
\end{equation}
It holds that $D_N[t+\sigma_N+\tau_N+s] = 0$ for all $s+\tau_N>H$ by the assumption that $d_i[t] = 0$ for $t>H$. Thus the last term in \eqref{eq:u_n_d} is given by 
\begin{multline*}
  \sum_{s=0}^{H-\tau_N}D_N[t+\sigma_N+\tau_N+s]\prod_{j = 2}^{s+1}g \\
  = \sum_{s=\tau_N}^{H}D_N[t+\sigma_N+s]\prod_{j=2}^{s-\tau_N+1}g = \sum_{s=\tau_N}^{H}D_N[t+\sigma_N+s]\prod_{j=\tau_N+2}^{s+1}g.
\end{multline*}
Using the definition for $D_i[t]$, the first two terms in \eqref{eq:u_n_d} are equal to
\[
  \sum_{i=1}^N \sum_{s=0}^{\tau_N-1}d_i[t+s] + \sum_{i=1}^{N-1}\sum_{s = \sigma_i+\tau_N}^{\sigma_{i+1}+\tau_N-1}\sum_{j=1}^i d_j[t+s-\sigma_j].
\]
Collecting all $d_k$ terms for a given $k$ gives
\[
  \sum_{s=\sigma_k}^{\sigma_k+\tau_N-1}d_k[t+s-\sigma_k] + \sum_{i=k}^{N-1}\sum_{s = \sigma_i+\tau_N}^{\sigma_{i+1}+\tau_N-1} d_k[t+s-\sigma_k] = \sum_{s=\sigma_k}^{\sigma_{N+1}-1}d_k[t+s-\sigma_k].
\]
And thus the first two terms in \eqref{eq:u_n_d} are equal to,
\[
   \sum_{i=1}^N \sum_{s=\sigma_i}^{\sigma_{N+1}-1}d_i[t+s-\sigma_i]
 = \sum_{i=1}^{N}\sum_{s = \sigma_i}^{\sigma_{i+1}-1}\sum_{j=1}^i d_j[t+s-\sigma_j] = \sum_{i=1}^{N}\sum_{d = \sigma_i}^{\sigma_{i+1}-1}D_i[t+d].
\]
Thus the total effect of the planned  disturbances in the expression for $u_N[t]$ in \eqref{eq:u_N_proof} is
\[
   \sum_{s=1}^{\bar\tau}d_{i}[t-s]+\sum_{i=1}^{N}\sum_{s = \sigma_i}^{\sigma_{i+1}-1}D_i[t+s] + \sum_{s=\tau_N}^{H} D_N[t+\sigma_{N}+s]\prod_{j=\tau_N+2}^{s+1}g.
\]
So $u_N[t]$ is given by
\begin{multline*}
  u_N[t] =   -\frac{X}{r} \Big[\sum_{i=1}^{N}\Big(y_i[t]+\sum_{s=\bar\tau+1}^{\tau_i+\bar\tau} u_i[t-s] + \sum_{s=1}^{\bar\tau}d_{i}[t-s]+ \sum_{s = \sigma_i}^{\sigma_{i+1}-1}D_i[t+s]\Big) \\
  +\sum_{s =1}^{\bar\tau}u_N[t-s] 
  + \sum_{s=\tau_N}^{H} D_N[t+\sigma_{N}+s]\prod_{j=2}^{d+1}g\Big].
\end{multline*}
Which can be expressed as
\[
  u_N[t] = -\frac{X}{r}\left[m_N + \sum_{s=\tau_N+1}^{H} D_N[t+\sigma_{N}+s]\prod_{j=2}^{d+1}g  \right].
\]

\subsubsection{Change of variables.}
The structured controller is synthesized for dynamics on the form in \eqref{eq:dyn_proof}, while the plant model is on the form in \eqref{eq:first_delayed_des}. However, there exists a simple change of variables that allows us to transform between the two models.
Consider the synthesis dynamics
\[
  y_i[t+1] = y_i[t] +  b_iu_i[t-\tau_i-\bar\tau] - c_{i}( u_{i-1}[t-\bar\tau] -  d_i[t-\bar\tau]).
\]
Let
\[\hat b_1 = b_1, \quad \hat b_i = \frac{b_i}{c_i}\hat b_{i-1}, \quad
  z_1 = y_1, \quad  z_i = \frac{\hat b_{i-1}}{c_i}y_i,
\]
and 
\[
  \hat u_i = \hat b_i u_i, \quad \hat d_1 = c_1d_1, \ \ \hat d_i = \hat b_{i-1} d_i \ i\geq 2.
\]
Then the dynamics in \eqref{eq:first_delayed_des} are transformed to
\begin{equation}\label{eq:transformed_dyn}
  z_i[t+1] = z_i[t] + \hat u_i[t-\tau_i-\bar\tau] - \hat u_{i-1}[t-\bar\tau] +\hat d_i[t-\bar\tau].
\end{equation}
This follows trivially for node 1. For node $i$, we get by replacing $u_{i-1}$ with $1/\hat{b}_{i-1}\cdot \hat{u}_{i-1}$
\[
  y_i[t+1] = y_i[t] +  b_iu_i[t-\tau_i-\bar\tau] - c_i/\hat{b}_{i-1}\cdot u_{i-1}[t-\bar\tau] -  c_id_i[t-\bar\tau]).
\]
Which can be rewritten as
\[
  \hat{b}_{i-1}/c_iy_i[t+1] = \hat{b}_{i-1}/c_i y_i[t] +  \hat{b}_{i-1}b_i/c_iu_i[t-\tau_i-\bar\tau] - \hat{u}_{i-1}[t-\bar\tau] -  \hat{b}_{i-1}d_i[t-\bar\tau]).
\]
Applying the suggested change of variables gives the dynamics in \eqref{eq:transformed_dyn}.
For the cost parameters it follows that
\[
  ru_N^2 = \frac{r}{\hat b_N^2}\hat u_N^2,
  \quad q_iy_i = \frac{q_ic_i^2}{\hat b_{i-1}^2}z_i,\quad  i\geq 2.
\]

This change of variables is implemented in  \cref{alg:init} on lines 4 and 8, and in \cref{alg:impl} on lines 1-2 and 19.

\subsection{LQ with known disturbance}
Here we give the derivation of a LQ controller with feed-forward. That is we consider the problem

\[\begin{aligned}
  \minimize \quad & \sum_{t=0}^{\infty} x[t]^TQx[t] + u[t]^TRu[t] \\
  \st \quad &x[t+1] = Ax[t] + Bu[t] + v[t] \\
  & x[0] \text{ and } v[t] \text{ given}.
\end{aligned}\]
This is a well studied problem when $v[t] = 0$, see for example \cite{bertsekas2012dynamic}, and we only consider the extension due to the planned disturbance $v[t]$.

Assume that $v[t] = 0$ for all $t>N$ and $R$ is positive definite. Let $S$ be the solution to the algebraic Riccati equation 
\[
  S =A^TSA- A^TSB(B^TSB+R)^{-1}B^TSA + Q.
\]
Note that $S$ will be symmetric.
The cost to go from time $N+1$ is given by $x[N+1]^TSx[N+1]$, and the optimal $u[N]$ is given by the minimizer for the cost to go from time $t = N$:
\begin{multline}\label{eq:cost_to_go_N}
  x[N]^TQx[N] + u[N]^TRu[N] + x[N+1]^TSx[N+1] = \\
  x[N]^TQx[N] + u[N]^TRu[N] + \\
  (Ax[N] + Bu[N] + v[N])^TS(Ax[N] + Bu[N] + v[N]).
\end{multline}
Collecting all terms which has $u[N]$ in them gives
\[
  u[N]^TRu[N] + 2(Ax[N]+ v[N])^TSBu[N] + u[N]^TB^TSBu[N].
\]
The problem is convex, and differentiating with respect to $u$ gives that the optimal $u$ is given by
\[\begin{aligned}
  2(B^TSB+R)u[N] &= -2B^TS(Ax[N]+v[N]) \\
  u[N]  & = -(B^TSB+R)^{-1}B^TS(Ax[N]+v[N]).
\end{aligned}\]
Let $\Pi[N] = Sv[N]$. It holds that $u[N] = Kx[N] +K_v\Pi[N]$, where $K$ and $K_v$ are as in \eqref{eq:LQ_K_expr}.
 Inserting the expression for $u[N]$ into \eqref{eq:cost_to_go_N} and only considering terms that depend on $x[N]$ gives for the cost to go:
\begin{equation}\label{eq:cost_to_go_N_simp}
\begin{aligned}
  &x^T[N]Qx[N] + (Ax[N]+v[N])^T\Big[SB(B^TSB+R)^{-T}R(B^TSB+R)^{-1}B^TS\\
  &\qquad+S - 2SB(B^TSB+R)^{-1}B^TS\\
  & \qquad +SB(B^TSB+R)^{-T}B^TSB(B^TSB+R)^{-1}B^TS\Big](Ax[N]+v[N]) = \\
  &x^T[N]Qx[N]\hspace{-2pt}  + \hspace{-2pt} (Ax[N]\hspace{-2pt}+\hspace{-2pt}v[N])^T\left[S\hspace{-2pt}-\hspace{-2pt}SB(B^TSB+R)^{-1}B^TS \right](Ax[N]\hspace{-2pt}+\hspace{-2pt}v[N]).
\end{aligned}
\end{equation}
The terms containing only $x[N]$ simplifies to $x^T[N]Sx[N]$. For the terms containing $x[N]$ and $v[N]$ we get
\[\begin{aligned}
  2v^T[N](S-SB(B^TSB+R)^{-1}B^TS)Ax[N]& \\
  =2v^T[N]S(A+BK)x[N]& \\
  = 2\Pi[N]^T(A+BK)x[N]&.
\end{aligned}\]
Thus the cost to go at time $N-1$ is given by
\[
  x[N-1]^TQx[N-1] + u[N-1]^TRu[N-1] + x[N]^TSx[N] + 2\Pi[N]^T(A+BK)x[N].
\]
Now assume that the cost to go for some $t$, $t\leq N-1$, is given by
\begin{equation}\label{eq:cost_to_go}
\begin{aligned}
  x[t]^TQx[t] + u[t]^TRu[t] + x[t+1]^TSx[t+1] + 2\Pi[t+1]^T(A+BK)x[t+1]  &\\
  =x[t]^TQx[t] + u[t]^TRu[t] + (Ax[t] + Bu[t] + v[t])^TS(Ax[t] + Bu[t] + v[t])&\\
   \qquad \qquad + 2\Pi[t+1]^T(A+BK)(Ax[t] + Bu[t] + v[t])&.
\end{aligned}
\end{equation}
The assumption holds for $t = N -1$ by the previous calculations.
Differentiating w.r.t $u[t]$ gives
\[
\begin{aligned}
  2(B^TSB+&R)u[t] = -2B^TS(Ax[N]+v[N]) -2B^T(A+BK)^T\Pi[t+1] \\
  \Rightarrow u[t]   &= -(B^TSB+R)^{-1}B^T(SAx[t]+Sv[t] + (A+BK)^T\Pi[t+1])\\
  & =Kx[t] + K_v\Big(Sv[t] + (A+BK)^T\Pi[t+1]\Big) .
\end{aligned}
\]
Letting $\Pi[t] =(A+BK)^T\Pi[t+1]+ Sv[t]$ gives that
\begin{equation}\label{eq:optimal_u_ecc}
  u[t]  = Kx[t] + K_v\Pi[t],
\end{equation}
 as long as the cost to go is given by \eqref{eq:cost_to_go}. 

Now we consider the cost to go in \eqref{eq:cost_to_go}. Every term that was in the cost to go for $t=N$ in \eqref{eq:cost_to_go_N_simp} will remain, that is the term $x[t]^TSx[t]$ and \mbox{$2v^T[t]S(A+BK)x[t]$}. However new terms will be added due to the addition of $\Pi[t+1]$ to the expression for $u[t]$ and the new term 
$$2\Pi[t+1](A+BK)x[t+1].$$
in \eqref{eq:cost_to_go} compared to \eqref{eq:cost_to_go_N}.
We ignore the term $2\Pi[t+1](A+BK)x[t+1]$ for now, and focus on the effect of $\Pi[t+1]$ in the expression for $u$.
The resulting effect on \eqref{eq:cost_to_go} for terms that include $x[t]$ are given by (where the first term is due to $u[t]^TRu[t]$, the second is due to $(Bu[t])^TSBu[t]$, and the third is due to $(Bu[t])^TSAx[t]$)
\[\begin{aligned}
  &2\Pi[t+1]^T(A+BK)K_v^TRKx[t] \\
  +&2\Pi[t+1]^T(A+BK)K_v^T  B^TSB  Kx[t]\\
   +&2\Pi[t+1]^T(A+BK)K_v^T  B^TSAx[t] = \\
  -&2\Pi[t+1]^T(A+BK)K_v^TB^TSAx[t]\\
   +&2\Pi[t+1]^T(A+BK)K_v^T   B^TSAx[t] = 0.
\end{aligned}\]
Where we have used that $(R+B^TSB)K = -B^TSA$.
The effect of the new term $2\Pi[t+1](A+BK)x[t+1]$ in terms of $x[t]$ is given by
\[
  2\Pi[t+1]^T(A+BK)(A+BK)x[t].
\]
So the total effect of the disturbances on the cost to go is given by 
\[2v[t]^TS(A+BK)x[t] +  2\Pi[t+1]^T(A+BK)(A+BK)x[t]
 = 2\Pi[t]^T(A+BK)x[t], 
\]
and thus the cost to go is given on the assumed form in \eqref{eq:cost_to_go} for $t-1$ as well. Thus \eqref{eq:cost_to_go} and \eqref{eq:optimal_u_ecc} holds for $0\leq t\leq N$.

\printbibliography
\end{document}